\documentclass[reqno,12pt]{amsart}
\usepackage[iso-8859-7]{inputenc}
\usepackage{amsmath,amsfonts,amssymb,amsthm,rotating}
\usepackage[T1]{fontenc}
\usepackage{enumerate}
\theoremstyle{plain}
\newtheorem{thm}{Theorem}

\newcommand{\slide}[1]
{\vspace*{22cm}
\begin{rotate}{90}
\begin{minipage}[l]{22cm}
\setlength{\parskip}{.5cm}
#1
\end{minipage}
\end{rotate}
\pagebreak
}

\begin{document}
\title{On Invariants and Forbidden Sets}
\author[Frank J. Palladino]{Frank J. Palladino}
\address{Department of Mathematics, University of Rhode Island,Kingston, RI 02881-0816, USA;}
\email{frank@math.uri.edu}
\date{April 1, 2010}

\begin{abstract}
\noindent 
We introduce six new algebraic invariants for rational difference equations. We use these invariants to perform a reduction of order in each case. This reduction of order allows us to find forbidden sets in each case. These six cases include two linear fractional rational difference equations of order greater than one. 
In all six cases, we give a closed form solution for all initial conditions which are not in the forbidden set. In all six cases, the initial conditions and parameters are assumed to be arbitrary complex numbers.
\end{abstract}
\maketitle

\section{Introduction}
Invariants have been used in several cases to provide insight into the behavior of rational difference equations, see [1]-[4], [6]-[10], [12]-[14], [16]-[22], [24], [25], and [38]-[40]. In this article, we introduce six new algebraic invariants for six second order rational difference equations. Two of the six second order rational difference equations are linear fractional rational difference equations. We use these invariants to perform a reduction of order in each case. This reduction of order allows us to find forbidden sets in each case. In each case, the second order rational difference equation is transformed, via the invariant, to a family of first order linear fractional rational difference equations. 
Since both the forbidden set, and the closed form solution are known for first order linear fractional rational difference equations, we use this information to find the forbidden set and closed form solution for the equations given in Section 4. For the reader's convenience, the forbidden set and closed form solution for first order linear fractional rational difference equations are presented in Section 3. In all six cases, we give a closed form solution for all initial conditions which are not in the forbidden set. In all six cases, the initial conditions and parameters are assumed to be arbitrary complex numbers.
In Section 2, we give some background information on some well known examples of invariants and forbidden sets for rational difference equations. 

\section{Background on invariants and forbidden sets for rational difference equations}

When building an understanding of invariants for rational difference equations it is helpful to use the following equations as prototypes:
\begin{equation}
x_{n+1}=\frac{\alpha + x_n}{x_{n-1}},\quad n=0,1,2,\dots,
\end{equation}
\begin{equation}
x_{n+1}=\frac{\alpha + x_n +x_{n-1} }{x_{n-2}},\quad n=0,1,2,\dots,
\end{equation}
with $\alpha>0$ and positive initial conditions.
Equation (1) is known by the cognomen Lyness' equation. Equation (2) is known by the cognomen Todd's equation and has also been referred to as the third order Lyness' recurrence, see for example \cite{gsm3}.
Equations (1) and (2) are discussed in the following references: [2], [4], [6]-[8], [12], [13], [15]-[26], and [38]-[40]. In \cite{klr} it is first shown for Equations (1) and (2) that we have an algebraic invariant in each case, particularly for Equation (1) we have:
$$\left(\alpha + x_{n-1}+ x_{n}\right)\left(1+\frac{1}{x_{n-1}}\right)\left(1+\frac{1}{x_{n}}\right)= constant.$$
For Equation (2) we have the algebraic invariant,
$$\left(\alpha + x_{n-2} + x_{n-1} + x_{n}\right)\left(1+\frac{1}{x_{n}}\right)\left(1+\frac{1}{x_{n-1}}\right)\left(1+\frac{1}{x_{n-2}}\right)= constant.$$\par
Now, in the case of Equation (1), it has been shown in the unpublished paper of Zeeman \cite{z} that the map induced by Equation (1) is conjugated to a rotation.
Moreover, in \cite{gsm3}, it has been shown that the three dimensional case given by Equation (2) is also conjugated to a rotation. It has been shown, see \cite{z} and \cite{gsm3}, that in both Equations (1) and (2), the phase space of the associated dynamical system is foliated by invariant curves. 
These invariant curves, which comprise the leaves of the foliation, are algebraic curves which degenerate into isolated points in some places.   \par
The Lyness invariants can be generalized for the following $k^{th}$ order rational difference equation, sometimes called the $k^{th}$ order Lyness equation:
\begin{equation}
x_{n+1}=\frac{\alpha + \sum^{k-1}_{i=0}x_{n-i} }{x_{n-k}},\quad n=0,1,2,\dots.
\end{equation}
with $\alpha\geq 0$ and positive initial conditions. With this generalization we obtain the following algebraic invariant:
$$\left(\prod^{k}_{i=0}\left( \frac{1}{x_{n-i}} +1 \right)\right)\left( \alpha + \sum^{k}_{i=0}x_{n-i} \right)= constant.$$
Significant work has been done on this equation, see for example \cite{br2} and \cite{gsm5}. Further invariants are given in \cite{gki} for $k$ sufficiently large. 
Recently geometric objects have been used effectively to obtain information about rational difference equations with invariants. A good example of this is the use of the Lie symmetry of the associated map by A. Cima, A. Gasull, and V. Ma\~{n}osa in \cite{gsm3} and \cite{gsm4}.
Another technique, which makes use of algebraic geometry, can be found in \cite{ag}.\par
The forbidden set of a rational difference equation is the set of initial conditions which eventually map to a singularity. Finding such sets has become a topic of recent interest in the literature, see for example [5], [11], [27]-[37]. Few techniques for determining forbidden sets are known. One such technique is the use of a semiconjugate factorization, see [29], [30], [32], and [34]-[37] for more on this topic.
In Section 4, we will find invariants for the given second order difference equations which allow us to perform a reduction of order in each case. In each case, the second order rational difference equation is transformed, via the invariant, to a family of first order linear fractional rational difference equations. 
Since the forbidden set and closed form solution is known for first order linear fractional rational difference equations, we use this information to find the forbidden set and closed form solution for the equations given in Section 4. For the reader's convenience, the forbidden set and closed form solution for first order linear fractional rational difference equations are presented in Section 3.

\section{The Riccati Difference Equation}
In this section, we briefly summarize the known results on the Riccati difference equation, see \cite{gks}, \cite{riccati} and \cite{kulenovicladas}. The branch cut for the complex square root will be taken to be the negative real numbers for the remainder of the article. 
\begin{thm}
 Consider the rational difference equation
$$x_{n+1}=\frac{\alpha + \beta x_{n}}{A+ Bx_{n}},\quad n=0,1,\dots,$$
where the parameters $\alpha , \beta , A, B$ and the initial condition $x_{0}$ are complex numbers.
There are seven possibilities.
\begin{enumerate}
\item Suppose $A=B=0$, then $\mathfrak{F}=\mathbb{C}$.
\item Suppose $B=0$, and $A\neq 0$, then $x_{n+1}=\frac{\alpha + \beta x_{n}}{A}$. So $\mathfrak{F}=\emptyset$ and $x_{n}=\frac{\beta^{n}x_{0}}{A^{n}}+\sum^{n}_{i=1}\frac{\alpha \beta^{n-1}}{A^{n}}$ for all $n\geq 1$.
\item Suppose $B\neq 0$ and $\alpha B - \beta A = 0$, then the forbidden set, $\mathfrak{F}= \{\frac{-A}{B}\}$ and $x_{n}=\frac{\beta}{B}$ for all $n\geq 1$ whenever $x_{0}\not\in \{\frac{-A}{B}\}$.
\item Suppose $B\neq 0$, $\alpha B - \beta A \neq 0$, and $\beta +A =0$, then the forbidden set, $\mathfrak{F}= \{\frac{-A}{B}\}$. Furthermore, $x_{2n+1}=\frac{\alpha + \beta x_{0}}{A+ B x_{0}}$ for all $n\geq 0$, and $x_{2n}=x_{0}$ for all $n\geq 0$, whenever $x_{0}\not\in  \{\frac{-A}{B}\}$.
\item Suppose $B\neq 0$, $\alpha B - \beta A \neq 0$, $\beta +A \neq 0$, and $\frac{\beta A - \alpha B}{(\beta + A)^{2}}\in \mathbb{C}\setminus [\frac{1}{4},\infty)$, and let $$w_{-}=\frac{1-\sqrt{1-4\left(\frac{\beta A - \alpha B}{(\beta + A)^{2}}\right)}}{2}\quad and \quad w_{+}=\frac{1+\sqrt{1-4\left(\frac{\beta A - \alpha B}{(\beta + A)^{2}}\right)}}{2},$$
then the forbidden set, $\mathfrak{F}= \left\{ \frac{\beta +A}{B}\left(\frac{w^{n-1}_{+}-w^{n-1}_{-}}{w^{n}_{+}-w^{n}_{-}}\right)w_{+}w_{-} - \frac{A}{B}|n\in\mathbb{N}\right\}$. Furthermore 
$$x_{n}=\frac{\beta +A}{B}\left(\frac{(\frac{Bx_{0}+A}{\beta + A}-w_{-})w^{n+1}_{+}-(w_{+}-\frac{Bx_{0}+A}{\beta + A})w^{n+1}_{-}}{(\frac{Bx_{0}+A}{\beta + A}-w_{-})w^{n}_{+}-(w_{+}-\frac{Bx_{0}+A}{\beta + A})w^{n}_{-}}\right) - \frac{A}{B},\quad for\; all\; n\in\mathbb{N},$$
whenever $x_{0}\not\in \left\{ \frac{\beta +A}{B}\left(\frac{w^{n-1}_{+}-w^{n-1}_{-}}{w^{n}_{+}-w^{n}_{-}}\right)w_{+}w_{-} - \frac{A}{B}|n\in\mathbb{N}\right\}$.
\item Suppose $B\neq 0$, $\alpha B - \beta A \neq 0$, $\beta +A \neq 0$, and $\frac{\beta A - \alpha B}{(\beta + A)^{2}}=\frac{1}{4}$,
then the forbidden set, $\mathfrak{F}= \left\{ \frac{\beta +A}{B}\left(\frac{n-1}{2n}\right) - \frac{A}{B}|n\in\mathbb{N}\right\}$. Furthermore 
$$x_{n}=\frac{\beta +A}{B}\left(\frac{1+(\frac{2Bx_{0}+2A}{\beta + A}-1)(n+1)}{2+2(\frac{2Bx_{0}+2A}{\beta + A}-1)n}\right) - \frac{A}{B},\quad for\; all\; n\in\mathbb{N},$$
whenever $x_{0}\not\in \left\{ \frac{\beta +A}{B}\left(\frac{n-1}{2n}\right) - \frac{A}{B}|n\in\mathbb{N}\right\}$.
\item Suppose $B\neq 0$, $\alpha B - \beta A \neq 0$, $\beta +A \neq 0$, and $R=\frac{\beta A - \alpha B}{(\beta + A)^{2}}\in (\frac{1}{4},\infty)$, 
let $\phi = \arccos{\left(\frac{1}{2}\sqrt{\frac{1}{R}}\right)}$,then the forbidden set, $\mathfrak{F}= \left\{ \frac{\beta +A}{2B}\left(1-\sqrt{4R-1}\cot{\left(n\phi\right)}\right) - \frac{A}{B}|n\in\mathbb{N}\right\}$. Furthermore, for all $n\in\mathbb{N}$,
$$x_{n}=\frac{\beta +A}{B}\left(\sqrt{R}\right)\left(\frac{\sqrt{4R-1}\cos{\left((n+1)\phi\right)}-(\frac{2Bx_{0}+2A}{\beta + A}-1)\sin{\left((n+1)\phi\right)}}{\sqrt{4R-1}\cos{\left(n\phi\right)}-(\frac{2Bx_{0}+2A}{\beta + A}-1)\sin{\left(n\phi\right)}}\right) - \frac{A}{B},$$
whenever $x_{0}\not\in \left\{ \frac{\beta +A}{2B}\left(1-\sqrt{4R-1}\cot{n\phi}\right) - \frac{A}{B}|n\in\mathbb{N}\right\}$.
\end{enumerate}

\end{thm}

\section{Using invariants to find forbidden sets}
Here we present six rational difference equations, each of order 2, which possess algebraic invariants. The invariants allow for a reduction of order in each case so that the dynamics of the equation can be described by either a family of Riccati equations, or a family of linear equations. 
The examples here have nice algebraic properties which are conducive to this type of approach. A remaining question of importance which is left to further work is whether this approach can be adapted to yield forbidden sets and explicit closed form solutions for other rational equations. \par
In the remainder of the article we make the notational convention of representing the set of inital conditions as a set in $\mathbb{C}^{2}$ with coordinates $(z_{0},z_{-1})$ this will be important for the remainder of the article, especially as it is needed to accurately describe the forbidden sets. Moreover, to accommodate the large formulae necessary to give a complete description of the forbidden sets, the forbidden sets are included in figures 1 and 2.

\begin{thm}
Consider the rational difference equation,
\begin{equation}
 z_{n+1}=\frac{z_{n}}{1+B z_{n-1}-B z_{n}},\quad n=0,1,\dots,
\end{equation}
 with $B\in \mathbb{C}\setminus \{0\}$ and with initial conditions $z_{0},z_{-1}\in \mathbb{C}$. Then the forbidden set, $\mathfrak{F}=S_{1}$, where $S_{1}$ is given in Figure 1. Also given $(z_{0},z_{-1})\notin \mathfrak{F}$ there are four possibilities:
\begin{enumerate}[i.]
\item $z_{0}=0$, in which case $z_{n}=0$ for all $n\geq 0$.
\item $z_{0}\neq 0$ and $z_{-1}=\frac{-1}{B}$, in which case $z_{2n+1}=\frac{-1}{B}$ for all $n\geq 0$ and $z_{2n}=\frac{-1}{2B+B^{2}z_{2n-2}}$ for all $n\geq 1$. Thus 
$$z_{2n}=\frac{n+2+nBz_{0}+Bz_{0}}{nB+B+nB^{2}z_{0}}-\frac{2}{B},\quad n\geq 0.$$
\item $z_{0}=\frac{-1}{B}$, in which case $z_{2n}=\frac{-1}{B}$ for all $n\geq 0$ and $z_{2n+1}=\frac{-1}{2B+B^{2}z_{2n-1}}$ for all $n\geq 0$. Thus
$$z_{2n+1}=\frac{n+3+nBz_{-1}+2Bz_{-1}}{nB+2B+nB^{2}z_{-1}+B^{2}z_{-1}}-\frac{2}{B},\quad n\geq -1.$$
\item $z_{0}\neq 0, \frac{-1}{B}$ and $z_{-1}\neq \frac{-1}{B}$, in which case $z_{n+1}=\frac{1+Bz_{n}}{C-B-B^{2}z_{n}}$ for all $n\geq 0$, where $C=\left(\frac{1}{z_{0}} + B \right)\left(1+ B z_{-1} \right)$. This implies the following:
\begin{enumerate}[a.]
\item If $\frac{B}{C}\in \mathbb{C}\setminus \left[\frac{1}{4},\infty\right)$, then call $C-B-B^{2}z_{0}-C\lambda_{1}=M_{1}$ and $C\lambda_{2}+B+B^{2}z_{0}-C=M_{2}$, and $$z_{n}=\frac{-C}{B^{2}}\left(\frac{M_{2}\lambda^{n+1}_{1}+M_{1}\lambda^{n+1}_{2}}{M_{2}\lambda^{n}_{1}+M_{1}\lambda^{n}_{2}}\right)+\frac{C}{B^{2}}-\frac{1}{B},\quad n\geq 0.$$
Where $$\lambda_{1}=\frac{1-\sqrt{1-\frac{4B}{C}}}{2},\quad and \quad \lambda_{2}=\frac{1+\sqrt{1-\frac{4B}{C}}}{2}.$$
\item If $C=4B$, then 
$$z_{n}=\frac{4+(n+1)\left(2-2Bz_{0} \right)}{nB^{2}z_{0}-2B-nB}+\frac{3}{B},\quad n\geq 0.$$
\item If $\frac{C}{B}\in (0,4)$, then call $\arccos\left(\sqrt{\frac{C}{4B}}\right)=\rho$, and for $n\geq 0$ we have,
$$z_{n}=\frac{-\sqrt{\frac{C}{B}}}{B}\left(\frac{\left(\sqrt{\frac{4B}{C}-1}\right)\cos\left((n+1)\rho\right)+(2w_{0}-1)\sin\left((n+1)\rho\right)}{\left(\sqrt{\frac{4B}{C}-1}\right)\cos\left(n\rho\right)+(2w_{0}-1)\sin\left(n\rho\right)}\right)+\frac{C-B}{B^{2}}.$$
Where $$w_{0}=\frac{-B^{2}z_{0}+C-B}{C}.$$
\end{enumerate}\end{enumerate}
\end{thm}
\begin{proof}
Let us begin by finding the forbidden set for our Equation (4). Let $\mathfrak{F}_{1}$ be the forbidden set with $z_{0}=0$. Let $\mathfrak{F}_{2}$ be the forbidden set with $z_{0}\neq 0$ and $z_{-1}=\frac{-1}{B}$. Let $\mathfrak{F}_{3}$ be the forbidden set with $z_{0}=\frac{-1}{B}$. Let $\mathfrak{F}_{4}$ be the forbidden set with $z_{0}\neq 0,\frac{-1}{B}$ and $z_{-1}\neq\frac{-1}{B}$.
Then the forbidden set $\mathfrak{F}=\bigcup^{4}_{i=1}\mathfrak{F}_{i}$. So we must find all $\mathfrak{F}_{i}$ with $1\leq i\leq 4$.\par

Let us begin with $\mathfrak{F}_{1}$. We have assumed $z_{0}=0$. Let us further assume that $z_{-1}\neq \frac{-1}{B}$, then $z_{1}$ is well defined and equal to $0$. Whenever $(z_{n},z_{n-1})=(0,0)$, then $z_{n+1}$ is well defined and equal to $0$. Thus, by induction, $z_{n}$ is well defined and equal to $0$ for all $n\in\mathbb{N}$.
Thus, if $z_{-1}\neq \frac{-1}{B}$, then $(0,z_{-1})\not\in \mathfrak{F}_{1}$. On the other hand, assume $z_{0}=0$ and $z_{-1}=\frac{-1}{B}$, then $z_{1}$ is not well defined and so $(0,\frac{-1}{B})\in \mathfrak{F}_{1}$. Thus, $\mathfrak{F}_{1}=\{(0,\frac{-1}{B})\}$.\par

Now, let us find $\mathfrak{F}_{2}$. In this case we have assumed $z_{0}\neq 0$ and $z_{-1}=\frac{-1}{B}$. Assume $z_{2n}$ is well defined for $0\leq n\leq N$, then we may make an induction argument with $z_{-1}$ as the base case.
Assume that $z_{2k+1}=\frac{-1}{B}$ for $k<N$, then, by our earlier assumption, $z_{2k+2}$ is well defined and
$$z_{2k+2}= \frac{-1}{2B+B^{2}z_{2k}}\neq 0.$$ So,
$$z_{2k+3}=\frac{z_{2k+2}}{1+B z_{2k+1}-B z_{2k+2}}= \frac{z_{2k+2}}{-B z_{2k+2}}= \frac{-1}{B}.$$
By this induction argument, $z_{2n+1}=\frac{-1}{B}$ and assuming $z_{2N+2}$ is well defined,
$$z_{2n+2}= \frac{-1}{2B+B^{2}z_{2n}},$$
for $0\leq n\leq N$. Call the forbidden set of the following first order difference equation $\hat{\mathfrak{F}}$,
$$x_{n+1}= \frac{-1}{2B+B^{2}x_{n}}, \quad n=0,1,2,\dots.$$
Now, suppose $z_{0}\neq 0$, $z_{-1}=\frac{-1}{B}$, and $z_{0}\not\in \hat{\mathfrak{F}}$, and assume that $z_{2n}$ is well defined for $n\leq N$, then $z_{2N+1}=\frac{-1}{B}$ and 
$$1+Bz_{2N}-Bz_{2N+1}=2 + Bz_{2N}=\frac{2B + B^{2}z_{2N}}{B}\neq 0,$$
since $z_{0}\not\in \hat{\mathfrak{F}}$. Thus, $z_{2n}$ is well defined for $n\leq N+1$. By induction, $z_{n}$ is well defined for all $n\in\mathbb{N}$. Thus, $(z_{0},\frac{-1}{B})\not\in \mathfrak{F}_{2}$.\par
Now, suppose $z_{0}\neq 0$, $z_{-1}=\frac{-1}{B}$, and $z_{0}\in \hat{\mathfrak{F}}$. Further assume for the sake of contradiction that $(z_{0},\frac{-1}{B})\not\in \mathfrak{F}_{2}$. Then, since $(z_{0},\frac{-1}{B})\not\in \mathfrak{F}_{2}$, $z_{n}$ is well defined for all $n\in\mathbb{N}$, but also 
$2B+B^{2}z_{2N}=0$ for some $N\in\mathbb{N}$, since $z_{0}\in \hat{\mathfrak{F}}$. But then
$$1+Bz_{2N}-Bz_{2N+1}=2 + Bz_{2N}=\frac{2B + B^{2}z_{2N}}{B}= 0.$$ This is a contradiction. Thus, $(z_{0},\frac{-1}{B})\in \mathfrak{F}_{2}$. So $\mathfrak{F}_{2}=\left(\hat{\mathfrak{F}}\setminus \{0\}\right)\times \{\frac{-1}{B}\}$.\par
Next, let us find $\mathfrak{F}_{3}$. In this case we have assumed $z_{0}=\frac{-1}{B}$. Assume $z_{2n-1}$ is well defined for $1\leq n\leq N$, then we may make an induction argument with $z_{0}$ as the base case.
Assume that $z_{2k}=\frac{-1}{B}$ for $k<N$, then, by our earlier assumption, $z_{2k+1}$ is well defined and
$$z_{2k+1}= \frac{-1}{2B+B^{2}z_{2k}}\neq 0.$$ So,
$$z_{2k+2}=\frac{z_{2k+1}}{1+B z_{2k}-B z_{2k+1}}= \frac{z_{2k+1}}{-B z_{2k+1}}= \frac{-1}{B}.$$
By this induction argument, $z_{2n}=\frac{-1}{B}$ and assuming $z_{2N+1}$ is well defined,
$$z_{2n+1}= \frac{-1}{2B+B^{2}z_{2n}},$$
for $0\leq n\leq N$. Call the forbidden set of the following first order difference equation $\hat{\mathfrak{F}}$,
$$x_{n+1}= \frac{-1}{2B+B^{2}x_{n}}, \quad n=0,1,2,\dots.$$
Now, suppose $z_{0}=\frac{-1}{B}$ and $z_{-1}\not\in \hat{\mathfrak{F}}$, and assume that $z_{2n-1}$ is well defined for $n\leq N$, then $z_{2N}=\frac{-1}{B}$ and 
$$1+Bz_{2N-1}-Bz_{2N}=2 + Bz_{2N-1}=\frac{2B + B^{2}z_{2N-1}}{B}\neq 0,$$
since $z_{-1}\not\in \hat{\mathfrak{F}}$. Thus, $z_{2n-1}$ is well defined for $n\leq N+1$. By induction, $z_{n}$ is well defined for all $n\in\mathbb{N}$. Thus, $(\frac{-1}{B},z_{-1})\not\in \mathfrak{F}_{3}$.\par
Now, suppose $z_{0}=\frac{-1}{B}$ and $z_{-1}\in \hat{\mathfrak{F}}$. Further assume for the sake of contradiction that $(\frac{-1}{B},z_{-1})\not\in \mathfrak{F}_{3}$. Then, since $(\frac{-1}{B},z_{-1})\not\in \mathfrak{F}_{3}$, $z_{n}$ is well defined for all $n\in\mathbb{N}$, but also 
$2B+B^{2}z_{2N-1}=0$ for some $N\in\mathbb{N}$, since $z_{-1}\in \hat{\mathfrak{F}}$. But then
$$1+Bz_{2N-1}-Bz_{2N}=2 + Bz_{2N-1}=\frac{2B + B^{2}z_{2N-1}}{B}= 0.$$ This is a contradiction. Thus, $(\frac{-1}{B},z_{-1})\in \mathfrak{F}_{3}$. So, $\mathfrak{F}_{3}= \{\frac{-1}{B}\} \times \hat{\mathfrak{F}}$.\par
Finally, let us find $\mathfrak{F}_{4}$. In this case we have assumed $z_{0}\neq 0,\frac{-1}{B}$ and $z_{-1}\neq \frac{-1}{B}$. Assume $z_{n}$ is well defined for $n\leq N$, then a simple induction argument shows $z_{n}\neq 0$ for $1\leq n\leq N$. Using this we get that for $n<N$,
$$\left(\frac{1}{z_{n}} + B \right)\left(1+ B z_{n-1} \right)= \left(\frac{1+B z_{n-1}}{z_{n}}\right) \left(1+ B z_{n} \right)=$$
$$ \left(\frac{1+B z_{n-1}-Bz_{n}+Bz_{n}}{z_{n}}\right) \left(1+ B z_{n} \right)=$$
$$\left(\frac{1+B z_{n-1}-Bz_{n}}{z_{n}} + B\right) \left(1+ B z_{n} \right)=\left(\frac{1}{z_{n+1}} + B \right)\left(1+ B z_{n} \right).$$
So,
$$\left(\frac{1}{z_{n}} + B \right)\left(1+ B z_{n-1} \right)=constant,$$
for $0\leq n\leq N$. Thus, $z_{n}\neq \frac{-1}{B}$ for $n\leq N$ and assuming $z_{N+1}$ is well defined, 
$$z_{n+1}= \frac{1+Bz_{n}}{C-B-B^{2}z_{n}},$$
for $0\leq n\leq N$. Where $C=\left(\frac{1}{z_{0}} + B \right)\left(1+ B z_{-1} \right)$.
Call the forbidden set of the following first order difference equation $\mathfrak{F}_{C}$,
$$x_{n+1}= \frac{1+Bx_{n}}{C-B-B^{2}x_{n}}, \quad n=0,1,2,\dots.$$
Note that the set $\mathfrak{F}_{C}$ changes depending on the value of $C$.
Now, suppose $z_{0}\neq 0,\frac{-1}{B}$, $z_{-1}\neq \frac{-1}{B}$, $C=\left(\frac{1}{z_{0}} + B \right)\left(1+ B z_{-1} \right)$, and $z_{0}\not\in \mathfrak{F}_{C}$, and assume that $z_{n}$ is well defined for $n\leq N$. Recall that we have shown that this implies $z_{n}\neq \frac{-1}{B}$ for $n\leq N$. Then 
$$1+Bz_{N-1}-Bz_{N}=\frac{C}{B+\frac{1}{z_{N}}}-Bz_{N}=\frac{C-B - B^{2}z_{N}}{B+\frac{1}{z_{N}}}\neq 0,$$
since $z_{0}\not\in \mathfrak{F}_{C}$. Thus, $z_{n}$ is well defined for $n\leq N+1$. By induction, $z_{n}$ is well defined for all $n\in\mathbb{N}$. Thus, $(z_{0},z_{-1})\not\in \mathfrak{F}_{4}$.\par
Now, suppose $z_{0}\neq 0,\frac{-1}{B}$, $z_{-1}\neq \frac{-1}{B}$, $C=\left(\frac{1}{z_{0}} + B \right)\left(1+ B z_{-1} \right)$, and $z_{0}\in \mathfrak{F}_{C}$. Further assume for the sake of contradiction that $(z_{0},z_{-1})\not\in \mathfrak{F}_{4}$. Then, since $(z_{0},z_{-1})\not\in \mathfrak{F}_{4}$, $z_{n}$ is well defined for all $n\in\mathbb{N}$, but also 
$C-B - B^{2}z_{N}=0$ for some $N\in\mathbb{N}$, since $z_{0}\in \mathfrak{F}_{C}$. But then
$$1+Bz_{N-1}-Bz_{N}=\frac{C}{B+\frac{1}{z_{N}}}-Bz_{N}=\frac{C-B - B^{2}z_{N}}{B+\frac{1}{z_{N}}}= 0.$$ This is a contradiction. Thus, $(z_{0},z_{-1})\in \mathfrak{F}_{4}$. So, 
$$\mathfrak{F}_{4}=\bigcup_{b\neq \frac{-1}{B}} \bigcup_{C\neq 0}  \left(\left(\left(\mathfrak{F}_{C}\setminus\left\{0,\frac{-1}{B}\right\}\right)\times \{b\}\right)\cap \left\{(a,b)|C=\left(\frac{1}{a} + B \right)\left(1+ B b \right)\right\}\right).$$
Reducing the above expression, we get
$$\mathfrak{F}_{4}= \bigcup_{C\neq 0} \left\{ \left(a,\frac{Ca-Ba-1}{B^{2}a+B}\right)| a\in\mathfrak{F}_{C}\setminus\left\{0,\frac{-1}{B}\right\}\right\}.$$
From the above characterization of $\mathfrak{F}_{1},\dots ,\mathfrak{F}_{4} $ and from the facts about the forbidden sets of the Riccati difference equation in Section 3, we get $\mathfrak{F}=S_{1}$, where $S_{1}$ is given in Figure 1. Now, let us describe the behavior when $(z_{0},z_{-1})\notin \mathfrak{F}$. Our analysis of this case will be broken into four subcases as shown in the statement of Theorem 2.
Let us first consider case (i). In this case $z_{0}=0$ and also, since $(z_{0},z_{-1})\notin \mathfrak{F}$, there will never be division by zero in our solution.
It is immediately clear from these two facts and from a basic induction argument that $z_{n}=0$ for all $n\geq 0$.
Now, let us consider case (ii). In this case $z_{0}\neq 0$ and $z_{-1}=\frac{-1}{B}$. Also, since $(z_{0},z_{-1})\notin \mathfrak{F}$, there will never be division by zero in our solution.
This allows us to prove by induction that $z_{n}\neq 0$ for all $n\geq 0$. The induction argument for this piece is straightforward and is omitted. Now, we will show by induction that $z_{2n+1}=\frac{-1}{B}$ for all $n\geq 0$.
Since $z_{0}\neq 0$, we have
$$z_{1}= \frac{z_{0}}{1+B z_{-1}-B z_{0}}= \frac{z_{0}}{-B z_{0}}= \frac{-1}{B}.$$
This provides the base case for our induction argument. Now, suppose $z_{2n-1}=\frac{-1}{B}$, then since $z_{n}\neq 0$ for all $n\geq 0$, we have
$$z_{2n+1}=\frac{z_{2n}}{1+B z_{2n-1}-B z_{2n}}= \frac{z_{2n}}{-B z_{2n}}= \frac{-1}{B}.$$
Thus, we have shown that $z_{2n+1}=\frac{-1}{B}$ for all $n\geq 0$. This fact and Equation (4) immediately yields,
$$z_{2n}=\frac{z_{2n-1}}{1+B z_{2n-2}-B z_{2n-1}}= \frac{-1}{2B+B^{2}z_{2n-2}}, \quad n=1,2,\dots.$$
Thus, the even terms are defined recursively by the above equation. Notice that since we have only rewritten the recursive Equation (4) we cannot have division by zero in this equation with our choice of initial conditions. In other words, $z_{2n}\neq \frac{-2}{B}$ for any $n\geq 0$.
Notice that this is a Riccati equation after a change of variables. Since we already know the closed form solution for any Riccati equation, we may obtain a closed form solution for $z_{2n}$, and thus a closed form solution for $z_{n}$ in this case. We use the known results for Riccati equations restated in Section 3 to obtain the closed form solutions in the statement of the theorem.\par
Now, let us consider case (iii). In this case, $z_{0}=\frac{-1}{B}$. Also, since $(z_{0},z_{-1})\notin \mathfrak{F}$, there will never be division by zero in our solution.
This allows us to prove by induction that $z_{n}\neq 0$ for all $n\geq 0$. The induction argument for this piece is straightforward and is omitted. Now, we will show by induction that $z_{2n}=\frac{-1}{B}$ for all $n\geq 0$.
Since $z_{1}\neq 0$, we have
$$z_{2}= \frac{z_{1}}{1+B z_{0}-B z_{1}}= \frac{z_{1}}{-B z_{1}}= \frac{-1}{B}.$$
This provides the base case for our induction argument. Now, suppose $z_{2n-2}=\frac{-1}{B}$, then since $z_{n}\neq 0$ for all $n\geq 0$, we have
$$z_{2n}=\frac{z_{2n-1}}{1+B z_{2n-2}-B z_{2n-1}}= \frac{z_{2n-1}}{-B z_{2n-1}}= \frac{-1}{B}.$$
Thus, we have shown that $z_{2n}=\frac{-1}{B}$ for all $n\geq 0$. This fact and Equation (4) immediately yields,
$$z_{2n+1}=\frac{z_{2n}}{1+B z_{2n-1}-B z_{2n}}= \frac{-1}{2B+B^{2}z_{2n-1}}, \quad n=0,1,2,\dots.$$
Thus, the odd terms are defined recursively by the above equation. Notice that since we have only rewritten the recursive Equation (4) we cannot have division by zero in this equation with our choice of initial conditions. In other words, $z_{2n-1}\neq \frac{-2}{B}$ for any $n\geq 0$.
Notice that this is a Riccati equation after a change of variables. Since we already know the closed form solution for any Riccati equation, we may obtain a closed form solution for $z_{2n+1}$, and thus a closed form solution for $z_{n}$ in this case.  We use the known results for Riccati equations restated in Section 3 to obtain the closed form solutions in the statement of the theorem.\par
Let us finally consider case (iv). In this case $z_{0}\neq 0, \frac{-1}{B}$ and $z_{-1}\neq \frac{-1}{B}$. Also, since $(z_{0},z_{-1})\notin \mathfrak{F}$, there will never be division by zero in our solution.
This allows us to prove by induction that $z_{n}\neq 0$ for all $n\geq 0$. The induction argument for this piece is straightforward and is omitted. Since there is never division by zero in our solution, and since $z_{n}\neq 0$ for all $n\geq 0$, the following algebraic computation is well defined.
$$\left(\frac{1}{z_{n}} + B \right)\left(1+ B z_{n-1} \right)= \left(\frac{1+B z_{n-1}}{z_{n}}\right) \left(1+ B z_{n} \right)=$$ $$\left(\frac{1+B z_{n-1}-Bz_{n}+Bz_{n}}{z_{n}}\right) \left(1+ B z_{n} \right)=$$ $$\left(\frac{1+B z_{n-1}-Bz_{n}}{z_{n}} + B\right) \left(1+ B z_{n} \right)=\left(\frac{1}{z_{n+1}} + B \right)\left(1+ B z_{n} \right).$$
Thus, we have the following algebraic invariant:
$$\left(\frac{1}{z_{n}} + B \right)\left(1+ B z_{n-1} \right)=constant.$$
For our fixed but arbitrary initial conditions with $z_{0}\neq 0, \frac{-1}{B}$ and $z_{-1}\neq \frac{-1}{B}$ let us denote $C=\left(\frac{1}{z_{0}} + B \right)\left(1+ B z_{-1} \right)$. Since $z_{0}\neq 0$, $C$ is well defined, and since $z_{0},z_{-1}\neq \frac{-1}{B}$, $C\neq 0$.
Since $C\neq 0$, this forces $z_{n}\neq \frac{-1}{B}$ for all $n\geq 0$. Now, we claim that since $(z_{0},z_{-1})\notin \mathfrak{F}$, $z_{n}\neq \frac{C-B}{B^{2}}$ for all $n\geq 0$. In the case where $C=B$ it follows from the fact that $z_{n}\neq 0$ for all $n\geq 0$. In the remaining case, suppose there were such an $N$, then:
$$\left(\frac{B^{2}}{C-B} + B \right)\left(1+ B z_{N-1} \right)=C.$$
This implies that $z_{N}=\frac{C-B}{B^{2}}$ and $z_{N-1}= \frac{C-2B}{B^{2}}$, but then $1+Bz_{N-1}-Bz_{N}= 0$. However, this contradicts the fact that $(z_{0},z_{-1})\notin \mathfrak{F}$. Thus, we have that $z_{n}\neq \frac{C-B}{B^{2}}$ for all $n\geq 0$.
Algebraic manipulations of our invariant yield the following:
$$z_{n}=\frac{1+Bz_{n-1}}{C-B-B^{2}z_{n-1}},\quad n\geq 1.$$
Since $z_{n}\neq \frac{C-B}{B^{2}}$ for all $n\geq 0$, this equation is well-defined for all $n\geq 1$. Thus the dynamics of $\{z_{n}\}^{\infty}_{n=-1}$ are given by a Riccati equation in this case. Since we already know the closed form solution for any Riccati equation, we may obtain a closed form solution for $\{z_{n}\}^{\infty}_{n=-1}$ in this case. We use the known results for Riccati equations restated in Section 3 to obtain the closed form solutions in the statement of the theorem.
\end{proof}

\begin{thm}
Consider the rational difference equation,
\begin{equation}
 z_{n+1}=\frac{z_{n-1}}{1+B z_{n}-B z_{n-1}},\quad n=0,1,\dots,
\end{equation}
 with $B\in \mathbb{C}\setminus \{0\}$ and with initial conditions $z_{0},z_{-1}\in \mathbb{C}$. Then the forbidden set, $\mathfrak{F}=S_{2}$, where $S_{2}$ is given in Figure 1.
 Also, given $(z_{0},z_{-1})\notin \mathfrak{F}$, there are four possibilities:
\begin{enumerate}[i.]
\item $z_{0}=0$, in which case $z_{2n}=0$ for all $n\geq 0$ and $z_{2n+1}=\frac{z_{2n-1}}{1-Bz_{2n-1}}$ for all $n\geq 0$. This implies 
$$z_{2n+1}=\frac{1}{-B}\left(\frac{1-(n+2)Bz_{-1}}{1-(n+1)Bz_{-1}}\right)+\frac{1}{B},\quad n\geq -1.$$
\item $z_{-1}=0$, in which case $z_{2n+1}=0$ for all $n\geq 0$ and $z_{2n}=\frac{z_{2n-2}}{1-Bz_{2n-2}}$ for all $n\geq 1$. This implies 
$$z_{2n}=\frac{1}{-B}\left(\frac{1-(n+1)Bz_{0}}{1-nBz_{0}}\right)+\frac{1}{B},\quad n\geq 0.$$
\item $z_{0}=\frac{-1}{B}$, in which case $z_{n}=\frac{-1}{B}$ for all $n\geq 0$.
\item $z_{0}\neq 0, \frac{-1}{B}$ and $z_{-1}\neq 0$, in which case $z_{n+1}=\frac{1}{Cz_{n}-B}$ for all $n\geq 0$, where $C=\left(\frac{1}{z_{0}} + B \right)\left(\frac{1}{z_{-1}} \right)$. This implies the following:
\begin{enumerate}[a.]
\item If $\frac{-C}{B^{2}}\in \mathbb{C}\setminus \left[\frac{1}{4},\infty\right)$, then $$z_{n}=\frac{-B}{C}\left(\frac{(B\lambda_{2}+Cz_{0}-B)\lambda^{n+1}_{1}+(B-Cz_{0}-B\lambda_{1})\lambda^{n+1}_{2}}{(B\lambda_{2}+Cz_{0}-B)\lambda^{n}_{1}+(B-Cz_{0}-B\lambda_{1})\lambda^{n}_{2}}\right)+\frac{B}{C},\quad n\geq 0.$$
Where $$\lambda_{1}=\frac{1-\sqrt{1+\frac{4C}{B^{2}}}}{2},\quad and \quad \lambda_{2}=\frac{1+\sqrt{1+\frac{4C}{B^{2}}}}{2}.$$
\item If $\frac{-C}{B^{2}}=\frac{1}{4}$, then 
$$z_{n}=\frac{-B}{C}\left(\frac{-B+(n+1)\left(2Cz_{0}-B \right)}{-2B+4nCz_{0}-2nB}\right)+\frac{B}{C},\quad n\geq 0.$$
\item If $\frac{-C}{B^{2}}\in \left(\frac{1}{4},\infty\right)$, then call $B\sqrt{\frac{-4C}{B^{2}}-1}=D$ and call $\arccos\left(\sqrt{\frac{-B^{2}}{4C}}\right)=\rho$ and for $n\geq 0$, we have
$$z_{n}=\sqrt{\frac{-B^{2}}{C}}\left(\frac{D\cos\left((n+1)\rho\right)+(B-2Cz_{0})\sin\left((n+1)\rho\right)}{BD\cos\left(n\rho\right)+(B^{2}-2CBz_{0})\sin\left(n\rho\right)}\right)+\frac{B}{C}.$$
\end{enumerate}
\end{enumerate}
\end{thm}
\begin{proof}
We begin by finding the forbidden set for our Equation (5). Let $\mathfrak{F}_{1}$ be the forbidden set with $z_{0}=0$. Let $\mathfrak{F}_{2}$ be the forbidden set with $z_{-1}=0$. Let $\mathfrak{F}_{3}$ be the forbidden set with $z_{0}=\frac{-1}{B}$. Let $\mathfrak{F}_{4}$ be the forbidden set with $z_{0}\neq 0,\frac{-1}{B}$ and $z_{-1}\neq 0$.
Then the forbidden set $\mathfrak{F}=\bigcup^{4}_{i=1}\mathfrak{F}_{i}$. So we must find all $\mathfrak{F}_{i}$ with $1\leq i\leq 4$.\par
Let us begin with $\mathfrak{F}_{1}$. We have assumed $z_{0}=0$. Assume $z_{n}$ is well defined for $0\leq n\leq 2N$, then by a simple induction argument $z_{2n}=0$ for $0\leq n\leq N$ and so assuming $z_{2N+1}$ is well defined,
$$z_{2n+1}= \frac{z_{2n-1}}{1-Bz_{2n-1}},$$
for $0\leq n\leq N$. Call the forbidden set of the following first order difference equation $\hat{\mathfrak{F}}$,
$$x_{n+1}= \frac{x_{n}}{1-Bx_{n}}, \quad n=0,1,2,\dots.$$
Now, suppose $z_{0}=0$ and $z_{-1}\not\in \hat{\mathfrak{F}}$, and assume that $z_{n}$ is well defined for $n\leq 2N$, then $z_{2N}=0$ and 
$$1+Bz_{2N}-Bz_{2N-1}=1 - Bz_{2N-1}\neq 0,$$
since $z_{-1}\not\in \hat{\mathfrak{F}}$. Thus, $z_{n}$ is well defined for $n\leq 2N+1$ and
$$z_{2N+1}= \frac{z_{2N-1}}{1-Bz_{2N-1}}.$$
So, 
$$1+Bz_{2N+1}-Bz_{2N}=1 + Bz_{2N+1}= \frac{1-Bz_{2N-1} + Bz_{2N-1}}{1-Bz_{2N-1}}= \frac{1}{1-Bz_{2N-1}}\neq 0,$$
since $1 - Bz_{2N-1}\neq 0$. Thus $z_{n}$ is well defined for $n\leq 2N+2$. By induction $z_{n}$ is well defined for all $n\in\mathbb{N}$. Thus $(0,z_{-1})\not\in \mathfrak{F}_{1}$.\par
Now suppose $z_{0}=0$ and $z_{-1}\in \hat{\mathfrak{F}}$. Further assume for the sake of contradiction that $(0,z_{-1})\not\in \mathfrak{F}_{1}$. Then, since $(0,z_{-1})\not\in \mathfrak{F}_{1}$, $z_{n}$ is well defined for all $n\in\mathbb{N}$, but also 
$1-Bz_{2N-1}=0$ for some $N\in\mathbb{N}$ since $z_{-1}\in \hat{\mathfrak{F}}$. But then
$$1+Bz_{2N}-Bz_{2N-1}=1 - Bz_{2N-1}= 0.$$ This is a contradiction. Thus $(0,z_{-1})\in \mathfrak{F}_{1}$. So $\mathfrak{F}_{1}=\{0\}\times \hat{\mathfrak{F}} $.\par
Next, let us find $\mathfrak{F}_{2}$. In this case, we have assumed $z_{-1}=0$. Assume $z_{n}$ is well defined for $0\leq n\leq 2N+1$, then by a simple induction argment $z_{2n+1}=0$ for $0\leq n\leq N$ and so assuming $z_{2N+2}$ is well defined,
$$z_{2n+2}= \frac{z_{2n}}{1-Bz_{2n}},$$
for $0\leq n\leq N$. Call the forbidden set of the following first order difference equation $\hat{\mathfrak{F}}$,
$$x_{n+1}= \frac{x_{n}}{1-Bx_{n}}, \quad n=0,1,2,\dots.$$
Now, suppose $z_{-1}=0$ and $z_{0}\not\in \hat{\mathfrak{F}}$ and assume that $z_{n}$ is well defined for $n\leq 2N+1$, then $z_{2N+1}=0$ and 
$$1+Bz_{2N+1}-Bz_{2N}=1 - Bz_{2N}\neq 0,$$
since $z_{0}\not\in \hat{\mathfrak{F}}$. Thus, $z_{n}$ is well defined for $n\leq 2N+2$ and
$$z_{2N+2}= \frac{z_{2N}}{1-Bz_{2N}}.$$
So,
$$1+Bz_{2N+2}-Bz_{2N+1}=1 + Bz_{2N+2}= \frac{1-Bz_{2N} + Bz_{2N}}{1-Bz_{2N}}= \frac{1}{1-Bz_{2N}}\neq 0,$$
since $1 - Bz_{2N}\neq 0$. Thus, $z_{n}$ is well defined for $n\leq 2N+3$. By induction $z_{n}$ is well defined for all $n\in\mathbb{N}$. Thus, $(z_{0},0)\not\in \mathfrak{F}_{2}$.\par
Now, suppose $z_{-1}=0$ and $z_{0}\in \hat{\mathfrak{F}}$. Further assume for the sake of contradiction that $(z_{0},0)\not\in \mathfrak{F}_{2}$. Then, since $(z_{0},0)\not\in \mathfrak{F}_{2}$, $z_{n}$ is well defined for all $n\in\mathbb{N}$, but also 
$1-Bz_{2N}=0$ for some $N\in\mathbb{N}$, since $z_{0}\in \hat{\mathfrak{F}}$. But then
$$1+Bz_{2N+1}-Bz_{2N}=1 - Bz_{2N}= 0.$$ This is a contradiction. Thus, $(z_{0},0)\in \mathfrak{F}_{2}$. So $\mathfrak{F}_{2}= \hat{\mathfrak{F}}\times \{0\} $.\par
Now, let us find $\mathfrak{F}_{3}$. We have assumed $z_{0}=\frac{-1}{B}$. Let us further assume that $z_{-1}\neq 0$, then $z_{1}$ is well defined and equal to $\frac{-1}{B}$. Whenever $(z_{n},z_{n-1})=(\frac{-1}{B},\frac{-1}{B})$ then $z_{n+1}$ is well defined and equal to $\frac{-1}{B}$. Thus, by induction, $z_{n}$ is well defined and equal to $\frac{-1}{B}$ for all $n\in\mathbb{N}$.
Thus, if $z_{-1}\neq 0$, then $(\frac{-1}{B},z_{-1})\not\in \mathfrak{F}_{3}$. On the other hand, assume $z_{0}=\frac{-1}{B}$ and $z_{-1}=0$. Then $z_{1}$ is not well defined, and so $(\frac{-1}{B},0)\in \mathfrak{F}_{3}$. Thus, $\mathfrak{F}_{3}=\{(\frac{-1}{B},0)\}$.\par
Finally, let us find $\mathfrak{F}_{4}$. In this case we have assumed $z_{0}\neq 0,\frac{-1}{B}$ and $z_{-1}\neq 0$. Assume $z_{n}$ is well defined for $n\leq N$, then a simple induction argument shows $z_{n}\neq 0$ for $n\leq N$, and so for $n < N$,
$$\left(\frac{1}{z_{n+1}} + B \right)\left(\frac{1}{z_{n}} \right)= \left(\frac{1+ B z_{n} - B z_{n-1}}{z_{n-1}} + B \right) \left(\frac{1}{ z_{n}} \right)=$$
$$ \left(\frac{1+ B z_{n}}{z_{n-1}} \right) \left(\frac{1}{ z_{n}} \right)=\left(\frac{1+ B z_{n}}{z_{n}} \right) \left(\frac{1}{ z_{n-1}} \right)=\left(\frac{1}{z_{n}} + B \right)\left(\frac{1}{z_{n-1}} \right).$$
So,
$$\left(\frac{1}{z_{n}} + B \right)\left(\frac{1}{z_{n-1}} \right)=constant,$$
for $0\leq n\leq N$. Thus, $z_{n}\neq \frac{-1}{B}$ for $n\leq N$ and assuming $z_{N+1}$ is well defined,
$$z_{n+1}= \frac{1}{Cz_{n}-B},$$
for $0\leq n\leq N$. Where $C=\left(\frac{1}{z_{0}} + B \right)\left(\frac{1}{z_{-1}} \right)$.
Call the forbidden set of the following first order difference equation $\mathfrak{F}_{C}$,
$$x_{n+1}= \frac{1}{Cx_{n}-B}, \quad n=0,1,2,\dots.$$
Note that the set $\mathfrak{F}_{C}$ changes depending on the value of $C$.
Now, suppose $z_{0}\neq 0,\frac{-1}{B}$, $z_{-1}\neq 0$, $C=\left(\frac{1}{z_{0}} + B \right)\left(\frac{1}{z_{-1}} \right)$, and $z_{0}\not\in \mathfrak{F}_{C}$, and assume that $z_{n}$ is well defined for $n\leq N$. Then, notice that $$z_{N-1}=\frac{1+Bz_{N}}{Cz_{N}}.$$
Thus $$1+Bz_{N}-Bz_{N-1}=1+Bz_{N}-\left(\frac{B+B^{2}z_{N}}{Cz_{N}}\right)=\frac{Cz_{N}+BCz^{2}_{N}-B-B^{2}z_{N}}{Cz_{N}}$$$$=\frac{(Cz_{N}-B)(Bz_{N}+1)}{Cz_{N}}\neq 0,$$
since $z_{0}\not\in \mathfrak{F}_{C}$ and $z_{n}\neq 0,\frac{-1}{B}$ for all $n\leq N$. Thus, $z_{n}$ is well defined for $n\leq N+1$. By induction, $z_{n}$ is well defined for all $n\in\mathbb{N}$. Thus $(z_{0},z_{-1})\not\in \mathfrak{F}_{4}$.\newline
Now, suppose $z_{0}\neq 0,\frac{-1}{B}$, $z_{-1}\neq 0$, $C=\left(\frac{1}{z_{0}} + B \right)\left(\frac{1}{z_{-1}} \right)$, and $z_{0}\in \mathfrak{F}_{C}$. Further assume for the sake of contradiction that $(z_{0},z_{-1})\not\in \mathfrak{F}_{4}$. Then, since $(z_{0},z_{-1})\not\in \mathfrak{F}_{4}$, $z_{n}$ is well defined for all $n\in\mathbb{N}$, but also 
$Cz_{N}-B=0$ for some $N\in\mathbb{N}$ since $z_{0}\in \mathfrak{F}_{C}$. But then
$$1+Bz_{N}-Bz_{N-1}=1+Bz_{N}-\left(\frac{B+B^{2}z_{N}}{Cz_{N}}\right)=\frac{(Cz_{N}-B)(Bz_{N}+1)}{Cz_{N}}= 0.$$ This is a contradiction. Thus, $(z_{0},z_{-1})\in \mathfrak{F}_{4}$. 
So, $$\mathfrak{F}_{4}=\bigcup_{b\neq 0} \bigcup_{C\neq 0}  \left(\left(\left(\mathfrak{F}_{C}\setminus\left\{0,\frac{-1}{B}\right\}\right)\times \{b\}\right)\cap \left\{(a,b)|C=\left(\frac{1}{a} + B \right)\left(\frac{1}{b} \right)\right\}\right).$$
Reducing the above expression, we get
$$\mathfrak{F}_{4}= \bigcup_{C\neq 0} \left\{ \left(a,\frac{Ba+1}{Ca}\right)| a\in\mathfrak{F}_{C}\setminus\left\{0,\frac{-1}{B}\right\}\right\}.$$

From the above characterization of $\mathfrak{F}_{1},\dots ,\mathfrak{F}_{4} $ and from the facts about the forbidden sets of the Riccati difference equation in Section 3, we get $\mathfrak{F}=S_{2}$, where $S_{2}$ is given in Figure 1.

Now, we will describe the behavior when $(z_{0},z_{-1})\notin \mathfrak{F}$. Our analysis of this case will be broken into four subcases as shown in the statement of Theorem 3.
Let us first consider case (i). In this case, $z_{0}=0$ and also since $(z_{0},z_{-1})\notin \mathfrak{F}$ there will never be division by zero in our solution.
It is immediately clear from these two facts and from a basic induction argument that $z_{2n}=0$ for all $n\geq 0$. This fact and Equation (5) immediately yields,
$$z_{2n+1}=\frac{z_{2n-1}}{1+Bz_{2n}-Bz_{2n-1}}=\frac{z_{2n-1}}{1-Bz_{2n-1}} , \quad n=0,1,2,\dots.$$
Thus, the odd terms are defined recursively by the above equation. Notice that since we have only rewritten the recursive Equation (5), we cannot have division by zero in this equation with our choice of initial conditions. In other words, $z_{2n+1}\neq \frac{1}{B}$ for any $n\geq 0$.
Notice that this is a Riccati equation after a change of variables. Since we already know the closed form solution for any Riccati equation, we may obtain a closed form solution for $z_{2n+1}$, and thus a closed form solution for $z_{n}$ in this case. We use the known results for Riccati equations restated in Section 3 to obtain the closed form solutions in the statement of the theorem.\par
Now, let us consider case (ii). In this case $z_{-1}=0$ and also since $(z_{0},z_{-1})\notin \mathfrak{F}$ there will never be division by zero in our solution.
It is immediately clear from these two facts and from a basic induction argument that $z_{2n+1}=0$ for all $n\geq 0$. This fact and Equation (5) immediately yields,
$$z_{2n}=\frac{z_{2n-2}}{1+Bz_{2n-1}-Bz_{2n-2}}=\frac{z_{2n-2}}{1-Bz_{2n-2}} , \quad n=1,2,\dots.$$
Thus, the even terms are defined recursively by the above equation. Notice that since we have only rewritten the recursive Equation (5), we cannot have division by zero in this equation with our choice of initial conditions. In other words, $z_{2n}\neq \frac{1}{B}$ for any $n\geq 0$.
Notice that this is a Riccati equation after a change of variables. Since we already know the closed form solution for any Riccati equation, we may obtain a closed form solution for $z_{2n}$, and thus a closed form solution for $z_{n}$ in this case. We use the known results for Riccati equations restated in Section 3 to obtain the closed form solutions in the statement of the theorem.\par
Now, let us consider case (iii). In this case, $z_{0}=\frac{-1}{B}$. Also, since $(z_{0},z_{-1})\notin \mathfrak{F}$, there will never be division by zero in our solution.
This allows us to prove by induction that $z_{n} =  \frac{-1}{B}$ for all $n\geq 0$. The case $z_{0}= \frac{-1}{B}$ provides the base case. Assume that $z_{n}= \frac{-1}{B}$, since there is never division by zero from Equation (5) we get,
$$ z_{n+1}=\frac{z_{n-1}}{1+B z_{n}-B z_{n-1}}= \frac{z_{n-1}}{-Bz_{n-1}}= \frac{-1}{B}.$$
Thus we have shown that $z_{n} =  \frac{-1}{B}$ for all $n\geq 0$.\par
Let us finally consider case (iv). In this case, $z_{0}\neq 0, \frac{-1}{B}$ and $z_{-1}\neq 0$. Also, since $(z_{0},z_{-1})\notin \mathfrak{F}$, there will never be division by zero in our solution.
This allows us to prove by induction that $z_{n}\neq 0$ for all $n\geq 0$. The induction argument for this piece is straightforward and is omitted. Since there is never division by zero in our solution and since $z_{n}\neq 0$ for all $n\geq 0$, the following algebraic computation is well defined.
$$\left(\frac{1}{z_{n+1}} + B \right)\left(\frac{1}{z_{n}} \right)= \left(\frac{1+ B z_{n} - B z_{n-1}}{z_{n-1}} + B \right) \left(\frac{1}{ z_{n}} \right)= \left(\frac{1+ B z_{n}}{z_{n-1}} \right) \left(\frac{1}{ z_{n}} \right)=$$
$$\left(\frac{1+ B z_{n}}{z_{n}} \right) \left(\frac{1}{ z_{n-1}} \right)=\left(\frac{1}{z_{n}} + B \right)\left(\frac{1}{z_{n-1}} \right).$$
Thus, we have the following algebraic invariant:
$$\left(\frac{1}{z_{n}} + B \right)\left(\frac{1}{z_{n-1}} \right)=constant.$$
For our fixed but arbitrary initial conditions with $z_{0}\neq 0, \frac{-1}{B}$ and $z_{-1}\neq 0$, let us denote $C=\left(\frac{1}{z_{0}} + B \right)\left(\frac{1}{ z_{-1}} \right)$. Since $z_{0}\neq 0$, $C$ is well defined and since $z_{0}\neq \frac{-1}{B}$, $C\neq 0$.
Since $C\neq 0$, this forces $z_{n}\neq \frac{-1}{B}$ for all $n\geq 0$. Now, we claim that since $(z_{0},z_{-1})\notin \mathfrak{F}$, $z_{n}\neq \frac{B}{C}$ for all $n\geq 0$. In the case where $C=-B^{2}$, it follows from the fact that $z_{n}\neq \frac{-1}{B}$ for all $n\geq 0$. In the remaining case, suppose there were such an $N$, then:
$$\left(\frac{C}{B} + B \right)\left(\frac{1}{z_{N-1}} \right)=C.$$
This implies that $z_{N}=\frac{B}{C}$ and $z_{N-1}= \frac{B^{2}+C}{CB}$, but then $1+Bz_{N}-Bz_{N-1}= 0$. However, this contradicts the fact that $(z_{0},z_{-1})\notin \mathfrak{F}$. Thus, we have that $z_{n}\neq \frac{B}{C}$ for all $n\geq 0$.
Algebraic manipulations of our invariant yield the following:
$$z_{n}=\frac{1}{Cz_{n-1}- B},\quad n\geq 1.$$
Since $z_{n}\neq \frac{B}{C}$ for all $n\geq 0$, this equation is well-defined for all $n\geq 1$. Thus, the dynamics of $\{z_{n}\}^{\infty}_{n=-1}$ are given by a Riccati equation in this case. Since we already know the closed form solution for any Riccati equation, we may obtain a closed form solution for $\{z_{n}\}^{\infty}_{n=-1}$ in this case. We use the known results for Riccati equations restated in Section 3 to obtain the closed form solutions in the statement of the theorem.
\end{proof}

\begin{thm}
Consider the rational difference equation,
\begin{equation}
 z_{n+1}=\frac{z^{2}_{n}+Bz_{n} - Bz_{n-1}}{z_{n-1}},\quad n=0,1,\dots,
\end{equation}
 with $B\in \mathbb{C}$ and with initial conditions $z_{0},z_{-1}\in \mathbb{C}$. Then the forbidden set, $\mathfrak{F}=S_{3}$, where $S_{3}$ is given in Figure 2.
Also, given $(z_{0},z_{-1})\notin \mathfrak{F}$, $z_{n+1}=Cz_{n}-B$ for all $n\geq 0$, where $C=\frac{z_{0} + B }{z_{-1}}$. This implies $$z_{n}=C^{n}z_{0}-\sum^{n}_{k=1}C^{k-1}B,\quad n\geq 1.$$
\end{thm}
\begin{proof}
Let us first consider the case where $(z_{0},z_{-1})\notin \mathfrak{F}$. In this case clearly $z_{n}\neq 0$ for $n\geq -1$ or else we would have division by zero.
Since $z_{n}\neq 0$ for all $n\geq -1$ the following algebraic computation is well defined:
$$\frac{z_{n+1}+B}{z_{n}}=\frac{\frac{z^{2}_{n}+Bz_{n} - Bz_{n-1}}{z_{n-1}}+B}{z_{n}}= \frac{z^{2}_{n}+Bz_{n}}{z_{n}z_{n-1}}= \frac{z_{n}+B}{z_{n-1}}.$$
Thus, we have the following algebraic invariant:
$$\frac{z_{n}+B}{z_{n-1}}=constant.$$
For our fixed but arbitrary initial conditions, let us denote $C=\frac{z_{0} + B }{z_{-1}}$. Since $z_{-1}\neq 0$, $C$ is well defined.
 Algebraic manipulations of our invariant yield the following:
$$z_{n}=Cz_{n-1}-B,\quad n\geq 1.$$
Thus, the dynamics of $\{z_{n}\}^{\infty}_{n=-1}$ are given by a linear equation in this case. Since we already know the closed form solution for any linear equation, we may obtain a closed form solution for $\{z_{n}\}^{\infty}_{n=-1}$ in this case. 
The resulting closed form solution is $$z_{n}=C^{n}z_{0}-\sum^{n}_{k=1}C^{k-1}B,\quad n\geq 1.$$
Now, we must find the forbidden set for our Equation (6). Let $\mathfrak{F}$ be the forbidden set.
Suppose $z_{0},z_{-1}\neq 0$, $C=\frac{z_{0} + B }{z_{-1}}\neq 1$, and $z_{0}\not\in \left\{\frac{B-BC^{n}}{C^{n}-C^{n+1}}|n\in\mathbb{N}\right\}$, and assume that $z_{n}$ is well defined for $n\leq N$. 
Then $$z_{N-1}\neq 0,$$
since $z_{0}\not\in \{\frac{B-BC^{n}}{C^{n}-C^{n+1}}|n\in\mathbb{N}\}$. Thus, $z_{n}$ is well defined for $n\leq N+1$. By induction, $z_{n}$ is well defined for all $n\in\mathbb{N}$. Thus, $(z_{0},z_{-1})\not\in \mathfrak{F}$.\par
Now, suppose $z_{0},z_{-1}\neq 0$, $C=\frac{z_{0} + B }{z_{-1}}\neq 1$, and $z_{0}\in \left\{\frac{B-BC^{n}}{C^{n}-C^{n+1}}|n\in\mathbb{N}\right\}$. Further assume for the sake of contradiction that $(z_{0},z_{-1})\not\in \mathfrak{F}$. Then since $(z_{0},z_{-1})\not\in \mathfrak{F}$, $z_{n}$ is well defined for all $n\in\mathbb{N}$, but also 
$z_{N}=0$ for some $N\in\mathbb{N}$, since $z_{0}\in \left\{\frac{B-BC^{n}}{C^{n}-C^{n+1}}|n\in\mathbb{N}\right\}$. This is a contradiction. Thus, $(z_{0},z_{-1})\in \mathfrak{F}$. \par
Suppose $z_{0},z_{-1}\neq 0$, $C=\frac{z_{0} + B }{z_{-1}}= 1$, and $z_{0}\not\in \left\{nB|n\in\mathbb{N}\right\}$, and assume that $z_{n}$ is well defined for $n\leq N$. 
Then $$z_{N-1}\neq 0,$$
since $z_{0}\not\in \{nB|n\in\mathbb{N}\}$. Thus, $z_{n}$ is well defined for $n\leq N+1$. By induction, $z_{n}$ is well defined for all $n\in\mathbb{N}$. Thus, $(z_{0},z_{-1})\not\in \mathfrak{F}$.\par
Now, suppose $z_{0},z_{-1}\neq 0$, $C=\frac{z_{0} + B }{z_{-1}}=1$, and $z_{0}\in \left\{nB|n\in\mathbb{N}\right\}$. Further assume for the sake of contradiction that $(z_{0},z_{-1})\not\in \mathfrak{F}$. Then, since $(z_{0},z_{-1})\not\in \mathfrak{F}$, $z_{n}$ is well defined for all $n\in\mathbb{N}$, but also 
$z_{N}=0$ for some $N\in\mathbb{N}$, since $z_{0}\in \left\{nB|n\in\mathbb{N}\right\}$. This is a contradiction. Thus, $(z_{0},z_{-1})\in \mathfrak{F}$. \par
Finally, suppose either $z_{0}=0$ or $z_{-1}=0$, then $(z_{0},z_{-1})\in \mathfrak{F}$. Thus, $\mathfrak{F}=S_{3}$, where $S_{3}$ is given in Figure 2.
\end{proof}

\begin{thm}
Consider the rational difference equation,
\begin{equation}
 z_{n+1}=\frac{z^{2}_{n} + B z_{n}}{ z_{n-1} + B},\quad n=0,1,\dots,
\end{equation}
with $B\in \mathbb{C}\setminus \{0\}$ and with initial conditions $z_{0},z_{-1}\in \mathbb{C}$. Then the forbidden set, $\mathfrak{F}=S_{4}$, where $S_{4}$ is given in Figure 2.
Also, given $(z_{0},z_{-1})\notin \mathfrak{F}$ there are two possibilities:
\begin{enumerate}[i.]
\item $z_{0}=0$, in which case $z_{n}=0$ for all $n\geq 0$;
\item $z_{0}\neq 0$, in which case $z_{n+1}=\frac{z_{n}+B}{C}$ for all $n\geq 0$, where $C=\frac{z_{-1}+B}{z_{0}}$. This implies $$z_{n}=\frac{z_{0}}{C^{n}}+\sum^{n}_{k=1}\frac{B}{C^{k}},\quad n\geq 1.$$
\end{enumerate}
\end{thm}
\begin{proof}
Let us first consider the case where $(z_{0},z_{-1})\notin \mathfrak{F}$. Our analysis of this case will be broken into two subcases as shown in the statement of Theorem 5.
Let us first consider case (i). In this case, $z_{0}=0$ and also since $(z_{0},z_{-1})\notin \mathfrak{F}$, there will never be division by zero in our solution.
It is immediately clear from these two facts and from a basic induction argument that $z_{n}=0$ for all $n\geq 0$. 

Now, let us consider case (ii). In this case $z_{0}\neq 0$. Also, since $(z_{0},z_{-1})\notin \mathfrak{F}$, there will never be division by zero in our solution.
Thus, $z_{n}\neq -B$ for all $n\geq -1$ or else we would have division by zero.
This allows us to prove by induction that $z_{n}\neq 0$ for all $n\geq 0$. The induction argument for this piece is straightforward and is omitted. Since there is never division by zero in our solution, and since $z_{n}\neq 0$ for all $n\geq 0$, the following algebraic computation is well defined.
$$\frac{z_{n}+B}{z_{n+1}}= \frac{(z_{n}+B)(z_{n-1}+B)}{z^{2}_{n}+Bz_{n}}= \frac{z_{n-1}+B}{z_{n}}.$$
Thus, we have the following algebraic invariant:
$$\frac{z_{n-1}+B}{z_{n}}=constant.$$
For our fixed but arbitrary initial conditions with $z_{0}\neq 0$, let us denote $C=\frac{z_{-1}+B}{z_{0}}$. Since $z_{0}\neq 0$, $C$ is well defined and since $z_{-1}\neq -B$, $C\neq 0$.
Algebraic manipulations of our invariant yield the following:
$$z_{n}=\frac{z_{n-1}+B}{C} ,\quad n\geq 1.$$
Thus, the dynamics of $\{z_{n}\}^{\infty}_{n=-1}$ are given by a linear equation in this case. Since we already know the closed form solution for any linear equation, we may obtain a closed form solution for $\{z_{n}\}^{\infty}_{n=-1}$ in this case.
The resulting closed form solution is $$z_{n}=\frac{z_{0}}{C^{n}}+\sum^{n}_{k=1}\frac{B}{C^{k}},\quad n\geq 1.$$
Now, we must find the forbidden set for our Equation (7). Let $\mathfrak{F}$ be the forbidden set.
Suppose $z_{0},z_{-1}\neq -B$, $C=\frac{z_{-1} + B }{z_{0}}\neq 1$, and $z_{0}\not\in \left\{\frac{B-BC^{n+1}}{C-1}|n\in\mathbb{N}\right\}$, and assume that $z_{n}$ is well defined for $n\leq N$. 
Then $$z_{N-1}\neq -B,$$
since $z_{0}\not\in \left\{\frac{B-BC^{n+1}}{C-1}|n\in\mathbb{N}\right\}$. Thus, $z_{n}$ is well defined for $n\leq N+1$. By induction, $z_{n}$ is well defined for all $n\in\mathbb{N}$. Thus, $(z_{0},z_{-1})\not\in \mathfrak{F}$.\par
Now, suppose $z_{0},z_{-1}\neq -B$, $C=\frac{z_{-1} + B }{z_{0}}\neq 1$, and $z_{0}\in \left\{\frac{B-BC^{n+1}}{C-1}|n\in\mathbb{N}\right\}$. Further assume for the sake of contradiction that $(z_{0},z_{-1})\not\in \mathfrak{F}$. Then, since $(z_{0},z_{-1})\not\in \mathfrak{F}$, $z_{n}$ is well defined for all $n\in\mathbb{N}$, but also 
$z_{N}=-B$ for some $N\in\mathbb{N}$, since $z_{0}\in \left\{\frac{B-BC^{n+1}}{C-1}|n\in\mathbb{N}\right\}$. This is a contradiction. Thus, $(z_{0},z_{-1})\in \mathfrak{F}$.\par
Suppose $z_{0},z_{-1}\neq -B$, $C=\frac{z_{-1} + B }{z_{0}}= 1$, and $z_{0}\not\in \left\{-nB-B|n\in\mathbb{N}\right\}$, and assume that $z_{n}$ is well defined for $n\leq N$. 
Then $$z_{N-1}\neq -B,$$
since $z_{0}\not\in \{-nB-B|n\in\mathbb{N}\}$. Thus, $z_{n}$ is well defined for $n\leq N+1$. By induction, $z_{n}$ is well defined for all $n\in\mathbb{N}$. Thus, $(z_{0},z_{-1})\not\in \mathfrak{F}$.\par
Now, suppose $z_{0},z_{-1}\neq -B$, $C=\frac{z_{-1} + B }{z_{0}}=1$, and $z_{0}\in \left\{-nB-B|n\in\mathbb{N}\right\}$. Further assume for the sake of contradiction that $(z_{0},z_{-1})\not\in \mathfrak{F}$. Then, since $(z_{0},z_{-1})\not\in \mathfrak{F}$, $z_{n}$ is well defined for all $n\in\mathbb{N}$, but also 
$z_{N}=-B$ for some $N\in\mathbb{N}$, since $z_{0}\in \left\{-nB-B|n\in\mathbb{N}\right\}$. This is a contradiction. Thus, $(z_{0},z_{-1})\in \mathfrak{F}$.\par
Finally, suppose either $z_{0}=-B$ or $z_{-1}=-B$, then $(z_{0},z_{-1})\in \mathfrak{F}$. Thus, 
$\mathfrak{F}=S_{4}$, where $S_{4}$ is given in Figure 2.
\end{proof}

\begin{thm}
Consider the rational difference equation,
\begin{equation}
 z_{n+1}=\frac{z_{n}z_{n-1}+Bz_{n}}{B + z_{n}},\quad n=0,1,\dots,
\end{equation}
 with $B\in \mathbb{C}\setminus \{0\}$ and with initial conditions $z_{0},z_{-1}\in \mathbb{C}$. Then the forbidden set, $\mathfrak{F}=S_{5}$, where $S_{5}$ is given in Figure 2. Also, given $(z_{0},z_{-1})\notin \mathfrak{F}$, $z_{n+1}=\frac{C}{z_{n}+B}$ for all $n\geq 0$, where $C=\left(z_{-1} + B \right)\left(z_{0}\right)$. This implies the following:
\begin{enumerate}[a.]
\item If $C=0$, then $z_{n}=0$ for all $n\geq 0$.
\item If $\frac{-C}{B^{2}}\in \mathbb{C}\setminus \left[\frac{1}{4},\infty\right)$, then $$z_{n}=B\left(\frac{(B\lambda_{2}-z_{0}-B)\lambda^{n+1}_{1}+(z_{0}+B-B\lambda_{1})\lambda^{n+1}_{2}}{(B\lambda_{2}-z_{0}-B)\lambda^{n}_{1}+(z_{0}+B-B\lambda_{1})\lambda^{n}_{2}}\right)-B,\quad n\geq 0.$$
Where $$\lambda_{1}=\frac{1-\sqrt{1+\frac{4C}{B^{2}}}}{2},\quad and \quad \lambda_{2}=\frac{1+\sqrt{1+\frac{4C}{B^{2}}}}{2}.$$
\item If $\frac{-C}{B^{2}}=\frac{1}{4}$, then 
$$z_{n}=B\left(\frac{B+(n+1)\left(2z_{0}+B \right)}{2B+4nz_{0}+2nB}\right)-B,\quad n\geq 0.$$
\item If $\frac{-C}{B^{2}}\in \left(\frac{1}{4},\infty\right)$, then call $B\sqrt{\frac{-4C}{B^{2}}-1}=D$ and call $\arccos\left(\sqrt{\frac{B^{2}}{-4C}}\right)=\rho$, and for $n\geq 0$, we have
$$z_{n}=B\sqrt{\frac{-C}{B^{2}}}\left(\frac{D\cos\left((n+1)\rho\right)+(B+2z_{0})\sin\left((n+1)\rho\right)}{D\cos\left(n\rho\right)+(B+2z_{0})\sin\left(n\rho\right)}  \right)  -B.$$
\end{enumerate}\end{thm}
\begin{proof}
Let us first consider the case where $(z_{0},z_{-1})\notin \mathfrak{F}$. In this case clearly $z_{n}\neq -B$ for $n\geq 0$, or else we would have division by zero.
Since $z_{n}\neq -B$ for all $n\geq 0$, the following algebraic computation is well defined:
$$z_{n+1}(z_{n}+B)=\left(\frac{z_{n}z_{n-1}+Bz_{n}}{B + z_{n}}\right)(z_{n}+B)= z_{n}(z_{n-1}+B).$$
Thus we have the following algebraic invariant:
$$z_{n}(z_{n-1}+B)=constant.$$
For our fixed but arbitrary initial conditions, let us denote $C=z_{0}(z_{-1}+B)$. 
 Algebraic manipulations of our invariant yield the following:
$$z_{n}=\frac{C}{z_{n-1}+ B},\quad n\geq 1.$$
Since $z_{n}\neq -B$ for all $n\geq 0$, this equation is well-defined for all $n\geq 1$. Thus the dynamics of $\{z_{n}\}^{\infty}_{n=-1}$ are given by a Riccati equation in this case. Since we already know the closed form solution for any Riccati equation, we may obtain a closed form solution for $\{z_{n}\}^{\infty}_{n=-1}$ in this case. We use the known results for Riccati equations restated in Section 3 to obtain the closed form solutions in the statement of the theorem.\par
Now, we must find the forbidden set for our Equation (8).
Let $\mathfrak{F}$ be the forbidden set. Assume $z_{n}$ is well defined for $n\leq N$, then $z_{n}\neq -B$ for $1\leq n\leq N-1$. Using this we get that for $n<N$,
$$z_{n+1}(z_{n}+B)=\left(\frac{z_{n}z_{n-1}+Bz_{n}}{B + z_{n}}\right)(z_{n}+B)= z_{n}(z_{n-1}+B).$$
So,
$$z_{n}(z_{n-1}+B)=constant.$$
for $0\leq n\leq N$. Thus, assuming $z_{N+1}$ is well defined, 
$$z_{n+1}= \frac{C}{z_{n}+B},$$
for $0\leq n\leq N$. Where $C=z_{0}(z_{-1}+B)$.
Call the forbidden set of the following first order difference equation $\mathfrak{F}_{C}$,
$$x_{n+1}= \frac{C}{x_{n}+B}, \quad n=0,1,2,\dots.$$
Note that the set $\mathfrak{F}_{C}$ changes depending on the value of $C$.
Now, suppose $C=z_{0}(z_{-1}+B)$ and $z_{0}\not\in \mathfrak{F}_{C}$, and assume that $z_{n}$ is well defined for $n\leq N$. Recall that we have shown that this implies $z_{n}\neq -B$ for $n < N$. Then 
$$B+z_{n}\neq 0,$$
since $z_{0}\not\in \mathfrak{F}_{C}$. Thus, $z_{n}$ is well defined for $n\leq N+1$. By induction, $z_{n}$ is well defined for all $n\in\mathbb{N}$. Thus, $(z_{0},z_{-1})\not\in \mathfrak{F}$.\par
Now, suppose $C=z_{0}(z_{-1}+B)$ and $z_{0}\in \mathfrak{F}_{C}$. Further assume for the sake of contradiction that $(z_{0},z_{-1})\not\in \mathfrak{F}$. Then, since $(z_{0},z_{-1})\not\in \mathfrak{F}$, $z_{n}$ is well defined for all $n\in\mathbb{N}$, but also 
$B+z_{N}=0$ for some $N\in\mathbb{N}$, since $z_{0}\in \mathfrak{F}_{C}$. This is a contradiction. Thus $(z_{0},z_{-1})\in \mathfrak{F}$. So, 
$$\mathfrak{F}=\bigcup_{b\in\mathbb{C}} \bigcup_{C\in\mathbb{C}}  \left(\left(\mathfrak{F}_{C}\times \{b\}\right)\cap \left\{(a,b)|C=a(b+B)\right\}\right).$$
Notice that $\left(\{0\}\times \mathbb{C}\right)\cap \mathfrak{F}=\emptyset$ since if $z_{0}=0$ then a simple induction argument tells us that $z_{n}=0$ for all $n\in\mathbb{N}$, so there will never be division by zero in such a case. 
So we may reduce the above expression as follows,
$$\mathfrak{F}= \bigcup_{C\in\mathbb{C}} \left\{ \left(a,\frac{C-Ba}{a}\right)| a\in\mathfrak{F}_{C}\setminus\left\{0\right\}\right\}.$$
From the above reduction, and from the facts about the forbidden sets of the Riccati difference equation in Section 3, we get $\mathfrak{F}=S_{5}$, where $S_{5}$ is given in Figure 2.
\end{proof}

\begin{thm}
Consider the rational difference equation,
\begin{equation}
 z_{n+1}=\frac{z_{n}z_{n-1}+Bz_{n-1} - Bz_{n}}{z_{n}},\quad n=0,1,\dots,
\end{equation}
 with $B\in \mathbb{C}\setminus\{0\}$ and with initial conditions $z_{0},z_{-1}\in \mathbb{C}$. Then the forbidden set, $\mathfrak{F}=S_{6}$, where $S_{6}$ is given in Figure 2. Also, given $(z_{0},z_{-1})\notin \mathfrak{F}$, $z_{n+1}=\frac{C - B z_{n}}{z_{n}}$ for all $n\geq 0$,
 where $C=z_{-1}(z_{0}+B)$. This implies the following:
\begin{enumerate}[a.]
\item If $C=0$, then $z_{n}=-B$ for all $n\geq 0$.
\item If $\frac{-C}{B^{2}}\in \mathbb{C}\setminus \left[\frac{1}{4},\infty\right)$, then $$z_{n}=-B\left(\frac{(B\lambda_{2}+z_{0})\lambda^{n+1}_{1}-(z_{0}+B\lambda_{1})\lambda^{n+1}_{2}}{(B\lambda_{2}+z_{0})\lambda^{n}_{1}-(z_{0}+B\lambda_{1})\lambda^{n}_{2}}\right),\quad n\geq 0.$$
Where $$\lambda_{1}=\frac{1-\sqrt{1+\frac{4C}{B^{2}}}}{2},\quad and \quad \lambda_{2}=\frac{1+\sqrt{1+\frac{4C}{B^{2}}}}{2}.$$
\item If $\frac{-C}{B^{2}}=\frac{1}{4}$, then 
$$z_{n}=-B\left(\frac{-B+(n+1)\left(2z_{0}+B \right)}{-2B+4nz_{0}+2nB}\right),\quad n\geq 0.$$
\item If $\frac{-C}{B^{2}}\in \left(\frac{1}{4},\infty\right)$, then call $B\sqrt{\frac{-4C}{B^{2}}-1}=D$ and $\arccos\left(\sqrt{\frac{B^{2}}{-4C}}\right)=\rho$, and for $n\geq 0$, we get
$$z_{n}=-B\sqrt{\frac{-C}{B^{2}}}\left(\frac{D\cos\left((n+1)\rho\right)+(-2z_{0}-B)\sin\left((n+1)\rho\right)}{D\cos\left(n\rho\right)+(-2z_{0}-B)\sin\left(n\rho\right)}\right).$$
\end{enumerate}
\end{thm}
\begin{proof}
Let us first consider the case where $(z_{0},z_{-1})\notin \mathfrak{F}$. In this case clearly $z_{n}\neq 0$ for $n\geq 0$, or else we would have division by zero.
Since $z_{n}\neq 0$ for all $n\geq 0$, the following algebraic computation is well defined:
$$z_{n}(z_{n+1}+B)=\left(\frac{z_{n}z_{n-1}+Bz_{n-1} - Bz_{n}}{z_{n}}+B\right)(z_{n})= \left(\frac{z_{n}z_{n-1}+Bz_{n-1}}{z_{n}}\right)(z_{n})= z_{n-1}(z_{n}+B).$$
Thus, we have the following algebraic invariant:
$$z_{n-1}(z_{n}+B)=constant.$$
For our fixed but arbitrary initial conditions, let us denote $C=z_{-1}(z_{0}+B)$. 
 Algebraic manipulations of our invariant yield the following:
$$z_{n}=\frac{C - B z_{n-1}}{z_{n-1}},\quad n\geq 1.$$
Since $z_{n}\neq 0$ for all $n\geq 0$, this equation is well-defined for all $n\geq 1$. Thus, the dynamics of $\{z_{n}\}^{\infty}_{n=-1}$ are given by a Riccati equation in this case. Since we already know the closed form solution for any Riccati equation, we may obtain a closed form solution for $\{z_{n}\}^{\infty}_{n=-1}$ in this case. We use the known results for Riccati equations restated in Section 3 to obtain the closed form solutions in the statement of the theorem.
Now, we must find the forbidden set for our Equation (9). 
Let $\mathfrak{F}$ be the forbidden set. Assume $z_{n}$ is well defined for $n\leq N$, then $z_{n}\neq 0$ for $1\leq n\leq N-1$. Using this we get that for $n<N$
$$z_{n}(z_{n+1}+B)=\left(\frac{z_{n}z_{n-1}+Bz_{n-1} - Bz_{n}}{z_{n}}+B\right)(z_{n})=$$$$\left(\frac{z_{n}z_{n-1}+Bz_{n-1}}{z_{n}}\right)(z_{n})= z_{n-1}(z_{n}+B).$$
So,
$$z_{n}(z_{n-1}+B)=constant.$$
for $0\leq n\leq N$. Thus, assuming $z_{N+1}$ is well defined, 
$$z_{n+1}= \frac{C-Bz_{n}}{z_{n}},$$
for $0\leq n\leq N$, where $C=z_{-1}(z_{0}+B)$.
Call the forbidden set of the following first order difference equation $\mathfrak{F}_{C}$,
$$x_{n+1}= \frac{C-Bx_{n}}{x_{n}}, \quad n=0,1,2,\dots.$$
Note that the set $\mathfrak{F}_{C}$ changes depending on the value of $C$.
Now, suppose $C=z_{-1}(z_{0}+B)$ and $z_{0}\not\in \mathfrak{F}_{C}$, and assume that $z_{n}$ is well defined for $n\leq N$. Recall that we have shown that this implies $z_{n}\neq 0$ for $n < N$. Then 
$$z_{N}\neq 0,$$
since $z_{0}\not\in \mathfrak{F}_{C}$. Thus $z_{n}$ is well defined for $n\leq N+1$. By induction, $z_{n}$ is well defined for all $n\in\mathbb{N}$. Thus, $(z_{0},z_{-1})\not\in \mathfrak{F}$.\newline
Now, suppose $C=z_{-1}(z_{0}+B)$ and $z_{0}\in \mathfrak{F}_{C}$. Further assume for the sake of contradiction that $(z_{0},z_{-1})\not\in \mathfrak{F}$. Then, since $(z_{0},z_{-1})\not\in \mathfrak{F}$, $z_{n}$ is well defined for all $n\in\mathbb{N}$, but also 
$z_{N}=0$ for some $N\in\mathbb{N}$, since $z_{0}\in \mathfrak{F}_{C}$. This is a contradiction. Thus, $(z_{0},z_{-1})\in \mathfrak{F}$. So 
$$\mathfrak{F}=\bigcup_{z_{-1}\in\mathbb{C}} \bigcup_{C\in\mathbb{C}}  \left(\left(\mathfrak{F}_{C}\times \{z_{-1}\}\right)\cap \left\{(z_{0},z_{-1})|C=z_{-1}(z_{0}+B)\right\}\right).$$
Notice that $\left(\{-B\}\times \mathbb{C}\right)\cap \mathfrak{F}=\emptyset$, since if $z_{0}=-B$, then a simple induction argument tells us that $z_{n}=-B$ for all $n\in\mathbb{N}$, so there will never be division by zero in such a case. 
So, we may reduce the above expression as follows,
$$\mathfrak{F}= \bigcup_{C\in\mathbb{C}} \left\{ \left(a,\frac{C}{a+B}\right)| a\in\mathfrak{F}_{C}\setminus\left\{-B\right\}\right\}.$$
From the above reduction, and from the facts about the forbidden sets of the Riccati difference equation in Section 3, we get $\mathfrak{F}=S_{6}$, where $S_{6}$ is given in Figure 2.
\end{proof}

\section{Conclusion}
We have introduced some new invariants for rational difference equations, including some invariants for certain cases of the second order rational difference equation of the form,
$$x_{n+1}=\frac{\alpha + \beta x_{n} + \gamma x_{n-1}}{A + B x_{n} + C x_{n-1}},\quad n=0,1,2,\dots.$$
This second order linear fractional rational difference equation has been studied extensively in the case of nonnegative parameters and nonnegative initial conditions.
See \cite{kulenovicladas} for more information about this case. However, the invariants we have found only apply in a region of the parameters where at least one of the parameters cannot be a nonnegative real number.
For this reason, these particular examples were overlooked in \cite{kulenovicladas} as well as in the subsequent literature.\par
What is particularly interesting about the presented cases is that we are able to use invariants we find to obtain both the forbidden set, and a closed form solution for our rational difference equations through reduction of order.

\vfill

\begin{center}


\slide{

$$S_{1}=\left\{\left(0,\frac{-1}{B}\right)\right\}\bigcup \left\{\left(\frac{-1}{B}\left(\frac{n+1}{n}\right),\frac{-1}{B}\right)|n\in\mathbb{N}\right\}\bigcup \left\{\left(\frac{-1}{B},\frac{-1}{B}\left(\frac{n+1}{n}\right)\right)|n\in\mathbb{N}\right\}\bigcup $$
$$\bigcup_{D\in\mathbb{C}\setminus [0,4]}\left\{\left(a,\frac{DBa-Ba-1}{B^{2}a+B}\right)\left|a\in\left\{\frac{-2}{B}\left(\frac{\left(1+\sqrt{1-\frac{4}{D}}\right)^{n-1}-\left(1-\sqrt{1-\frac{4}{D}}\right)^{n-1}}{\left(1+\sqrt{1-\frac{4}{D}}\right)^{n}-\left(1-\sqrt{1-\frac{4}{D}}\right)^{n}}\right)+\frac{D-1}{B}\right|n\in\mathbb{N}\right\}\setminus\left\{0,\frac{-1}{B}\right\}\right\}$$
$$\bigcup_{D\in (0,4)} \left\{\left(a,\frac{DBa-Ba-1}{B^{2}a+B}\right)\left|a\in\left\{\frac{-D}{B}\left(1-\sqrt{\frac{4}{D}-1}\cot\left(n\cdot \arccos\left(\frac{\sqrt{D}}{2}\right)\right)\right)+\frac{D-1}{B}\right|n\in\mathbb{N}\right\}\setminus\left\{0,\frac{-1}{B}\right\}\right\}$$
$$\bigcup \left\{\left(a,\frac{3Ba-1}{B^{2}a+B}\right)\left|a\in \left\{\frac{-2}{B}\left(\frac{n-1}{n}\right)+\frac{3}{B}\right|n\in\mathbb{N}\right\}\setminus\left\{0,\frac{-1}{B}\right\}\right\}.$$ 

$$S_{2}=\left\{\left(\frac{-1}{B},0\right)\right\}\bigcup \left\{\left(0,\frac{1}{Bn}\right)|n\in\mathbb{N}\right\}\bigcup \left\{\left(\frac{1}{Bn},0\right)|n\in\mathbb{N}\right\}\bigcup $$
$$\bigcup_{D\in\mathbb{C}\setminus \left((-\infty,\frac{-1}{4}]\cup\{0\}\right)}\left\{\left(a,\frac{Ba+1}{DB^{2}a}\right)\left|a\in\left\{\frac{2}{B}\left(\frac{\left(1+\sqrt{1+4D}\right)^{n-1}-\left(1-\sqrt{1+4D}\right)^{n-1}}{\left(1+\sqrt{1+4D}\right)^{n}-\left(1-\sqrt{1+4D}\right)^{n}}\right)+\frac{1}{DB}\right|n\in\mathbb{N}\right\}\setminus\left\{0,\frac{-1}{B}\right\}\right\}$$
$$\bigcup_{D\in (-\infty,\frac{-1}{4})} \left\{\left(a,\frac{Ba+1}{DB^{2}a}\right)\left|a\in\left\{\frac{-1}{2DB}\left(-1-\sqrt{-4D-1}\cot\left(n\cdot \arccos\left(\frac{1}{2\sqrt{-D}}\right)\right)\right)\right|n\in\mathbb{N}\right\}\setminus\left\{0,\frac{-1}{B}\right\}\right\}$$
$$\bigcup \left\{ \left(a,\frac{-4Ba-4}{B^{2}a}\right)\left| a\in\left\{\frac{-2n-2}{Bn}\right|n\in\mathbb{N}\right\}\setminus\left\{0,\frac{-1}{B}\right\}\right\}.$$
\hspace{9cm} Figure 1.

}


\slide{
$$S_{3}=\bigcup_{C\neq 1}\left\{\left(\frac{B-BC^{n}}{C^{n}-C^{n+1}},\frac{B-BC^{n+1}}{C^{n+1}-C^{n+2}}\right)|n\in\mathbb{N}\right\}\bigcup\left(\{0\}\times \mathbb{C}\right)\bigcup\left(\mathbb{C}\times \{0\}\right)\bigcup\left\{\left(nB,nB+B\right)|n\in\mathbb{N}\right\}.$$
$$S_{4}=\bigcup_{C\neq 1}\left\{\left(\frac{B-BC^{n+1}}{C-1},\frac{B-BC^{n+2}}{C-1}\right)|n\in\mathbb{N}\right\}\bigcup\left(\{-B\}\times \mathbb{C}\right)\bigcup\left(\mathbb{C}\times \{-B\}\right)\bigcup\left\{\left(-nB-B,-nB-2B\right)|n\in\mathbb{N}\right\}.$$
$$S_{5}= \{-B,-B\}\bigcup \left\{ \left(a,\frac{B^{2}}{-4a}-B\right)\left| a\in\left\{\frac{-Bn-B}{2n}\right|n\in\mathbb{N}\right\}\right\}\bigcup$$
$$\bigcup_{D\in\mathbb{C}\setminus \left((-\infty,\frac{-1}{4}]\cup\{0\}\right)}\left\{\left(a,\frac{DB^{2}}{a}-B\right)\left|a\in\left\{-2DB\left(\frac{\left(1+\sqrt{1+4D}\right)^{n-1}-\left(1-\sqrt{1+4D}\right)^{n-1}}{\left(1+\sqrt{1+4D}\right)^{n}-\left(1-\sqrt{1+4D}\right)^{n}}\right)-B\right|n\in\mathbb{N}\right\}\setminus\left\{0\right\}\right\}$$
$$\bigcup_{D\in (-\infty,\frac{-1}{4})} \left\{\left(a,\frac{DB^{2}}{a}-B\right)\left|a\in\left\{\frac{B}{2}\left(-1-\sqrt{-4D-1}\cot\left(n\cdot \arccos\left(\frac{1}{2\sqrt{-D}}\right)\right)\right)\right|n\in\mathbb{N}\right\}\setminus\left\{0\right\}\right\}.$$
$$S_{6}= \{0,0\}\bigcup \left\{ \left(a,\frac{-B^{2}}{4a+4B}\right)\left| a\in\left\{\frac{-Bn+B}{2n}\right|n\in\mathbb{N}\right\}\right\}\bigcup$$
$$\bigcup_{D\in\mathbb{C}\setminus \left((-\infty,\frac{-1}{4}]\cup\{0\}\right)}\left\{\left(a,\frac{DB^{2}}{a+B}\right)\left|a\in\left\{2DB\left(\frac{\left(1+\sqrt{1+4D}\right)^{n-1}-\left(1-\sqrt{1+4D}\right)^{n-1}}{\left(1+\sqrt{1+4D}\right)^{n}-\left(1-\sqrt{1+4D}\right)^{n}}\right)\right|n\in\mathbb{N}\right\}\setminus\left\{-B\right\}\right\}$$
$$\bigcup_{D\in (-\infty,\frac{-1}{4})} \left\{\left(a,\frac{DB^{2}}{a+B}\right)\left|a\in\left\{\frac{-B}{2}\left(1-\sqrt{-4D-1}\cot\left(n\cdot \arccos\left(\frac{1}{2\sqrt{-D}}\right)\right)\right)\right|n\in\mathbb{N}\right\}\setminus\left\{-B\right\}\right\}.$$
\hspace{9cm} Figure 2.

}

\end{center}

\par
\end{document}